\documentclass[twoside, a4paper, 12pt]{amsart}
\usepackage{fullpage}
\usepackage{amsfonts}
\usepackage{amssymb}
\usepackage{amsmath}
\usepackage{mathrsfs}
\usepackage{enumitem}
\usepackage{tikz}
\usepackage{float}
\theoremstyle{plain}
\newtheorem*{claim*}{Claim}
\newtheorem{thm}{Theorem}[section]
\newtheorem{cor}[thm]{Corollary}
\newtheorem{lemma}[thm]{Lemma}
\newtheorem{prop}[thm]{Proposition}
\theoremstyle{definition}
\newtheorem{defn}[thm]{Definition}
\newtheorem{ex}[thm]{Example}

\newtheorem{remark}[thm]{Remark}

\newtheorem{prob}[thm]{Open Problem}

\newcommand{\ar}{\mbox{${\mathcal{R}}$}}
\newcommand{\el}{\mbox{${\mathcal{L}}$}}
\newcommand{\h}{\mbox{${\mathcal{H}}$}}
\newcommand{\dee}{\mbox{${\mathcal{D}}$}}
\newcommand{\jay}{\mbox{${\mathcal{J}}$}}
\newcommand{\leqr}{\leq_{\mathcal{R}}}
\newcommand{\HR}{H_{\mathcal{R}}}
\newcommand{\HL}{H_{\mathcal{L}}}
\newcommand{\HH}{H_{\mathcal{H}}}
\newcommand{\HJ}{H_{\mathcal{J}}}
\newcommand{\Z}{\mathbb{Z}}
\newcommand{\N}{\mathbb{N}}

\begin{document}
\title{\large{The $\mathcal{R}$-height of semigroups and their bi-ideals}}
\author{Craig Miller}
\address{Department of Mathematics, University of York, UK, YO10 5DD}
\email{craig.miller@york.ac.uk}
\maketitle

\begin{abstract} 
The $\ar$-height of a semigroup $S$ is the height of the poset of $\ar$-classes of $S.$  Given a semigroup $S$ with finite $\ar$-height, we establish bounds on the $\ar$-height of bi-ideals, one-sided ideals and two-sided ideals; in particular, these substructures inherit the property of having finite $\ar$-height.  We then investigate whether these bounds can be attained. 
\end{abstract}
~

\section{\large{Introduction}\nopunct}
\label{sec:intro}

Green's relations, five equivalence relations based on mutual divisibility, are arguably the most important tools for analysing the structure of semigroups.
It is natural to consider how each of Green's relations on a semigroup relates to the corresponding relation on a subsemigroup.  In general, they bear little resemblance to each other.  However, Green's relations $\ar,$ $\mathcal{L}$ and $\mathcal{H}$ on a regular subsemigroup $T$ of a semigroup $S$ are the restrictions of the corresponding relations on $S.$  If $T$ is an ideal of $S,$ then the same is also true for Green's relations $\mathcal{D}$ and $\mathcal{J}$ \cite[Lemma 2.1]{East}.  In \cite{East:2020}, East and Higgins investigated the inheritance of Green's relations by subsemigroups in the presence of stability of elements.  East also, in \cite{East:2021}, characterised Green's relations on principal one-sided ideals of an arbitrary semigroup, and then applied this theory to full transformation monoids and symmetric inverse monoids.  

This article is concerned with the poset of $\ar$-classes of a semigroup.  (Two elements of a semigroup $S$ are $\ar$-related if they generate the same principal right ideal, and the set of $\ar$-classes of $S$ is a poset under the natural partial order associated with $\ar.$)  The height of this poset is called the $\ar$-height.  The term {\em $\ar$-height} first appeared in \cite{Fleischer:2017}, in which the $\ar$-height of certain finite transformation semigroups was considered.  However, as alluded to in \cite{Fleischer:2017}, this parameter plays an implicit role in the (right) Rhodes expansion of a semigroup, a powerful tool in complexity theory, as well as in similar constructions such as the cover expansion; see \cite{Birget:1984} and \cite[Chapter XII]{Eilenberg:1976}.

The purpose of this article is to compare the $\ar$-height of a semigroup with that of its bi-ideals, one-sided ideals and two-sided ideals.
A {\em bi-ideal} of a semigroup $S$ is a subset $B$ of $S$ such that $BS^1B\subseteq B.$  This notion generalises that of one-sided (and hence two-sided) ideals.  Bi-ideals were introduced by Good and Hughes in \cite{Good:1952}, and were then studied systematically by Lajos in \cite{Lajos:1969, Lajos:1971}. 

The paper is structured as follows.  In Section \ref{sec:preliminaries}, we first present the necessary preliminary material, and then provide some basic results regarding chains of $\ar$-classes.  In Section \ref{sec:bounds}, given a semigroup $S$ with finite $\ar$-height, we establish bounds on the $\ar$-height of arbitrary bi-ideals, one-sided ideals and two-sided ideals of $S.$  We then investigate in Section \ref{sec:attainbounds} whether these bounds can be attained.  We conclude with some open questions and potential directions for future research in Section~\ref{sec:conclusion}.

\section{\large{Preliminaries}\nopunct}
\label{sec:preliminaries}

\subsection{Definitions}

Throughout this section, $S$ will denote a semigroup. 
We denote by $S^1$ the monoid obtained from $S$ by adjoining an identity if necessary (if $S$ is already a monoid, then $S^1=S$).

A subset $A\subseteq S$ is said to be a {\em right ideal} of $S$ if $AS\subseteq A.$  Left ideals are defined dually, and an {\em ideal} of $S$ is a subset that is both a right ideal and a left ideal.

A right (resp.\ left) ideal $A$ of $S$ is said to be {\em generated by} $X\subseteq A$ if $A=XS^1$ (resp.\ $A=S^1X$).  A right (resp.\ left) ideal is said to be {\em finitely generated} if it can be generated by a finite set, and {\em principal} if it can be generated by a single element.

Principal (one-sided) ideals lead to the well-known Green's relations.  For this article, we only require Green's relation $\ar$ and its associated pre-order.  Green's preorder $\leqr$ is defined by
$$a\leqr b\Leftrightarrow aS^1\subseteq bS^1,$$
and this yields the relation $\ar$:
$$a\,\ar\,b\Leftrightarrow a\leqr b\text{ and }b\leqr a.$$
It is easy to see that $\ar$ is an equivalence relation on $S$ that is compatible with left multiplication (i.e.\ it is a {\em left congruence}).

When we need to distinguish between Green's relation $\ar$ on different semigroups, we will write the semigroup as a subscript, i.e.\ $\ar_S$ for $\ar$ on $S.$  For convenience, we will write $\leq_S$ rather than $\leq_{\mathcal{R}_S},$ and $a<_S b$ if $a\leq_Sb$ but $(a, b)\notin\ar_S.$

Following standard convention, we write $R_a$ (or $R_a^S$) to denote the $\ar$-class of an element $a\in S.$

Green's pre-order $\leqr$ induces a partial order on the set of $\mathcal{R}$-classes of $S,$ given by 
$$R_a\leq R_b\Leftrightarrow a\leqr b.$$
We note that the poset $S/\ar$ is isomorphic to the poset of principal right ideals of $S$ (under $\subseteq$).
The {\em $\mathcal{R}$-height} of $S$ is the height of the poset $S/\ar$; i.e.\ the supremum of the lengths of chains of $\ar$-classes (where the {\em length} of a chain is its cardinality).  We denote the $\ar$-height of $S$ by $\HR(S).$\par
The semigroup $S$ is said to be {\em right simple} if it has no proper right ideals, and $S$ is said to be {\em simple} if it has no proper ideals.  Certainly right simple semigroups are simple.

A right ideal $A$ of $S$ is said to be {\em minimal} if there is no right ideal of $S$ properly contained in $A.$  Minimal left ideals are defined dually.  Similarly, an ideal $A$ is called {\em minimal} if it does not contain any other ideal of $S.$  It turns out that, considered as semigroups, minimal right ideals are right simple \cite[Theorem 2.4]{Clifford:1948}, and minimal ideals are simple \cite[Theorem 1.1]{Clifford:1948}.  There is at most one minimal ideal of $S$; if it exists, we call it the {\em kernel} of $S$ and denote it by $K(S).$  On the other hand, $S$ may possess multiple minimal right ideals.  If $S$ has a minimal right ideal, then $K(S)$ is equal to the union of all the minimal right ideals \cite[Theorem 2.1]{Clifford:1948}.
A {\em completely simple} semigroup is a simple semigroup that possesses both minimal right ideals and minimal left ideals.

An element $a\in S$ is said to be {\em regular} if there exists $b\in S$ such that $a=aba.$  We denote the set of regular elements of $S$ by $\text{Reg}(S).$  The semigroup $S$ is said to be {\em regular} if $S=\text{Reg}(S).$  It turns that for every regular element $a\in S$ there exists $b\in S$ such that $a=aba$ and $b=bab$; in this case, the element $b$ is said to be an {\em inverse} of $a,$ and vice versa.  If $S$ is regular and each of its elements has a unique inverse, then $S$ is called {\em inverse}.

One of the most useful means of constructing semigroups is via a presentation.  We briefly discuss presentations here; we refer the reader to \cite{Howie:1995} for more information.

The {\em free semigroup} on a non-empty set $X,$ denoted by $X^+,$ is the set of all words over $X$ under the operation of concatenation.  A {\em presentation} is a pair $\langle X\,|\,R\rangle,$ where $X$ is a non-empty set and $R$ is a binary relation on $X^+.$  We call $R$ a set of {\em defining relations}, and we write $u=v$ for $(u, v)\in R.$  A semigroup $S$ is {\em defined by the presentation} $\langle X\,|\,R\rangle$ if it is isomorphic to $X^+/R^{\sharp},$ where $R^{\sharp}$ denotes the congruence generated by $R$ (that is, the smallest congruence on $X^+$ containing $R$).

Adjoining an empty word $\epsilon$ to $X^+$ yields the {\em free monoid} on $X,$ denoted by $X^{\ast}.$  Similarly, one can adjoin a zero to $X^+$ to obtain the {\em free semigroup with zero} on $X,$ denoted by $X^+_0.$  By replacing $X^+$ with $X^{\ast}$ or $X^+_0$ in the above definition of a presentation, one obtains the notion of a monoid presentation or a presentation of a semigroup with zero, respectively.

A presentation $\langle X\,|\,R\rangle$ can be viewed as a rewriting system, where each defining relation $u=v$ corresponds to a rewriting rule $u\to v.$
We define a binary relation $\to$ on $X^{\ast}$ by $w\to w^{\prime}$ if and only if $w=w_1uw_2, w^{\prime}=w_1vw_2$ for some $(u, v)\in R$ and $w_1, w_2\in X^{\ast}$.  We denote by $\xrightarrow{\ast}$ the reflexive and transitive closure of $\to$. 
The rewriting system $\langle X\,|\,R\rangle$ is {\em noetherian} if it is well-founded, i.e.\ if there is no infinite chain $w_1\to w_2\to\cdots$ of words from $X^{\ast}.$  The rewriting system is {\em confluent} if for any $w, w_1, w_2\in X^{\ast}$ with $w\xrightarrow{\ast}w_1$ and $w\xrightarrow{\ast}w_2$, there exists $z\in X^{\ast}$ such that $w_1\xrightarrow{\ast}z$ and $w_2\xrightarrow{\ast}z.$  If a rewriting system is both noetherian and confluent, it is said to be {\em complete}.  For a noetherian rewriting system, to determine confluence it suffices to consider the critical pairs.
A {\em critical pair} of $\langle X\,|\,R\rangle$ is a pair $(w_1, w_2)\in X^{\ast}\times X^{\ast}$ with $w_1\neq w_2$ for which there exists $w\in X^{\ast}$ such that $w\to w_1$ and $w\to w_2$.  A critical pair $(w_1, w_2)$ is said to {\em resolve} if there exists $z\in X^{\ast}$ such that $w_1\xrightarrow{\ast}z$ and $w_2\xrightarrow{\ast}z.$
A noetherian rewriting system is complete if and only if all critical pairs resolve \cite[Lemma 2.4]{Huet}.

A word in $X^{\ast}$ is called {\em irreducible} if it is does not contain a subword that forms the left-hand side of a rewriting rule.  If the rewriting system $\langle X\,|\,R\rangle$ is complete, then for any word $w\in X^{\ast}$ there is a unique irreducible word $z\in X^{\ast}$ with $w\xrightarrow{\ast}z$ \cite[Theorem 1.1.12]{Book:1993}.  In this case, the semigroup defined by $\langle X\,|\,R\rangle$ has a normal form consisting of all the irreducible words over $X.$

We refer the reader to \cite{Book:1993} for more information about rewriting systems.

\subsection{Elementary results}

We now provide a few basic results that will be useful in the next section.

\begin{lemma}
\label{max+min}
Let $S$ be a semigroup with finite $\mathcal{R}$-height.  Then $S$ has maximal and minimal $\mathcal{R}$-classes.  In particular, $S$ has a kernel.
\end{lemma}

\begin{proof}
Consider a chain of $\ar$-classes of $S$ of maximal length.  Since $S$ has finite $\ar$-height, this chain is finite and hence has both a maximal element and a minimal element.  By the maximality of the length of the chain, it follows that $S$ has a maximal $\ar$-class and a minimal $\ar$-class.
\end{proof}

\begin{lemma}
\label{R-height1}
Let $S$ be a semigroup.  Then $\HR(S)=1$ if and only if $S$ is a union of minimal right ideals.
\end{lemma}

\begin{proof}
Clearly $\HR(S)=1$ if and only if every $\ar$-class of $S$ is minimal, which is equivalent to $S$ being a union of minimal right ideals.
\end{proof}


\begin{lemma}
\label{bi-ideal chain}
Let $S$ be a semigroup with finite $\ar$-height, and let $B$ be a bi-ideal of $S.$  If $\HR(B)\geq n,$ then there exists
a chain
$$b_1<_B b_2<_B\dots<_B b_n$$
where $b_1\in K(S).$
\end{lemma}

\begin{proof}
Since $\HR(B)\geq n,$ there exists a chain
$$a_1<_B a_2<_B\dots<_B a_n.$$
Since $S$ has finite $\ar$-height, the kernel $K=K(S)$ exists by Lemma \ref{max+min}.
If $a_1\in K,$ we just set $b_i=a_i$ for all $i\in\{1, \dots, n\}.$  Suppose then that $a_1\notin K.$  Since $BKB\subseteq B\cap K,$ the intersection $B\cap K$ is non-empty.  Choose $u\in B\cap K$ and let $b_1=a_1u.$  Since $b_1\in B^2$ and $B$ is a subsemigroup of $S,$ we have that $b_1\in B.$  Also, we have that $b_1\in a_1K\subseteq K,$ since $K$ is an ideal of $S.$  Thus $b_1\in B\cap K.$  Clearly $b_1\leq_B a_1$.  In fact, we have that $b_1<_B a_1,$ for otherwise we would have $a_1\in b_1B\in K.$
Thus, setting $b_{i+1}=a_i$ for all $i\in\{1, \dots, n-1\}$ yields the desired chain.
\end{proof}

An element $a\in S$ is said to have a {\em local right identity} (in $S$) if $a\in aS$ (that is, $a=ab$ for some $b\in S$).  If $S$ is a monoid, a regular semigroup or a right simple semigroup, then every element has a local right identity.

\begin{lemma}
\label{lem:lri}
Let $S$ be a semigroup and let $B$ be a bi-ideal of $S.$  If $b, c\in B$ have local right identites in $B,$ then $b\leq_Bc$ if and only if $b\leq_Sc,$ and $b<_Bc$ if and only if $b<_Sc.$
\end{lemma}

\begin{proof}
If $b\leq_Bc$ then clearly $b\leq_Sc.$  Suppose that $b\leq_S c.$  Then $b=cs$ for some $s\in S^1.$  Now, by assumption, there exist $u, v\in B$ such that $b=bu$ and $c=cv.$  Then we have
$$b=bu=csu=c(vsu)\in cB,$$ 
using the fact that $B$ is a bi-ideal of $S.$  Thus $b\leq_Bc.$\par
Now, using the first part of the lemma, we have
$$b<_Bc\;\Leftrightarrow\;[b\leq_Bc\text{ and }c\not\leq_Bb]\;\Leftrightarrow\;[b\leq_Sc\text{ and }c\not\leq_Sb]\;\Leftrightarrow\;b<_Sc,$$
as required.
\end{proof}

\begin{cor}
\label{cor:bi-ideal,CS}
Let $S$ be a semigroup with a completely simple kernel $K=K(S),$ and let $B$ be a bi-ideal of $S.$  For any $b, c\in B\cap K,$ we have $b\leq_Bc$ if and only if $b\leq_Sc,$ and $b<_Bc$ if and only if $b<_Sc.$
\end{cor}

\begin{proof}
We show that every element of $B\cap K$ has a local right identity in $B,$ and the result then follows from Lemma \ref{lem:lri}.  
Consider $b\in B\cap K.$  Since $b\in K,$ we have that $b\,\ar_S\,b^2,$ so $b=b^2s$ for some $s\in S^1.$  Since $K$ is regular, there exists $x\in K$ such that $b=bxb.$  Thus $$b=b^2sxb=b(bsxb)\in bB,$$ so $b$ has a local right identity in $B.$
\end{proof}

\section{\large{Bounds on the $\mathcal{R}$-height of Bi-ideals}\nopunct}
\label{sec:bounds}

In general, the property of having finite $\ar$-height is not inherited by subsemigroups.  For example, the group of integers $\Z$ has $\ar$-height 1 but its subsemigroup $\N$ has infinite $\ar$-height.  Perhaps surprisingly, however, bi-ideals {\em do} inherit the property of having finite $\ar$-height.  In fact, given a semigroup $S$ with finite $\ar$-height, the following result establishes a bound on the $\ar$-height of an arbitrary bi-ideal of $S.$

\begin{thm}
\label{thm:bi-ideal,bound}
Let $S$ be a semigroup with finite $\mathcal{R}$-height, and let $B$ be a bi-ideal of $S.$  Then $$\HR(B)\leq 3n-1,$$ where $n$ is the maximum length of a chain of $\ar$-classes of $S$ that intersect $B.$
\end{thm}

\begin{proof}
Suppose for a contradiction that $\HR(B)\geq 3n.$  Then, by Lemma \ref{bi-ideal chain}, there exists a chain 
$$b_1<_B b_2<_B\cdots<_B b_{3n}$$
where $b_1\in K(S).$  For each $i\in\{1, \dots, 3n-1\}$ we have $b_i\in b_{i+1}B\subseteq b_{i+1}S,$ so $b_i\leq_S b_{i+1}.$  Thus we have a chain
$$b_1\leq_S b_2\leq_S\cdots\leq_S b_{3n}.$$
Since each $b_i$ belongs to $B,$ by assumption the chain 
$$R_{b_1}^S\leq_S R_{b_2}^S\leq_S\cdots\leq_S R_{b_{3n}}^S$$
of $\ar_S$-classes has size at most $n.$

We claim that $b_1<_S b_3.$  Indeed, suppose that $b_1\,\ar_S\,b_3.$  Since the $\ar_S$-classes in $K$ are minimal and $b_1^2\leq_S b_1,$ it follows that $b_1\,\ar_S\,b_1^2.$  Thus $b_3\,\ar_S\,b_1^2.$  We then have
$$b_2\in b_3B\subseteq(b_1^2S^1)B\subseteq b_1(BS^1B)\subseteq b_1B,$$
where for the final containment we use the fact that $B$ is a bi-ideal of $S.$   But then $b_2\leq_Bb_1,$ contradicting the fact that $b_1<_B b_2$, so we have established the claim.

It follows from the above claim that $n>1.$  Since $b_1<_S b_3,$ the chain 
$$R_{b_2}^S\leq_S R_{b_2}^S\leq_S\cdots\leq_S R_{b_{3n}}^S$$
has size at most $n-1.$
Since $\frac{3n-2}{n-1}=3+\frac{1}{n-1}>3,$ by the generalised pigeonhole principle there exist $i, j, k, l\in\{3, \dots, 3n\}$ with $i<j<k<l$ such that $b_i\,\ar_S\,b_j\,\ar_S\,b_k\,\ar_S\,b_l.$  Since $b_i\leq_S b_{i+1}\leq_S\dots\leq_S b_l,$ we deduce that $b_i\,\ar_S\,b_m$ for all $m\in\{i+1, \dots, l\}.$  In particular, we have $b_i\,\ar_S\,b_{i+3},$ and hence $b_{i+3}\in b_iS.$  Therefore, we have that
$$b_{i+2}\in b_{i+3}B\subseteq b_iSB\subseteq b_{i+1}BSB\subseteq b_{i+1}B,$$
using the fact that $B$ is a bi-ideal of $S.$  But this contradicts the fact that $b_{i+1}<_Bb_{i+2}$.  This completes the proof.
\end{proof}

\begin{cor}
\label{cor:bi-ideal,bound}
Let $S$ be a semigroup with finite $\mathcal{R}$-height, and let $B$ be a bi-ideal of $S.$  Then $\HR(B)\leq 3\HR(S)-1.$
\end{cor}

In the case that the kernel $K(S)$ is completely simple, we obtain a slightly shorter bound for the $\ar$-height of a bi-ideal than that given in Theorem \ref{thm:bi-ideal,bound}.

\begin{thm}
\label{thm:bi-ideal,CS,bound}
Let $S$ be a semigroup with finite $\mathcal{R}$-height whose kernel is completely simple, and let $B$ be a bi-ideal of $S.$  Then $$\HR(B)\leq 3n-2,$$ where $n$ is the maximum length of a chain of $\ar$-classes of $S$ that intersect $B.$
\end{thm}

\begin{proof}
Suppose for a contradiction that $\HR(B)\geq 3n-1.$  Then, by Lemma \ref{bi-ideal chain}, there exists a chain
$$b_1<_B b_2<_B\cdots<_B b_{3n-1}$$
where $b_1\in K(S).$  Then we have a chain
$$b_1\leq_S b_2\leq_S\cdots\leq_S b_{3n-1}.$$
By assumption, the chain 
$$R_{b_1}^S\leq_S R_{b_2}^S\leq_S\cdots\leq_S R_{b_{3n-1}}^S$$
has size at most $n.$  We cannot have $b_1\,\ar_S\,b_2,$ since that would imply that $b_1\,\ar_B\,b_2$ by Corollary \ref{cor:bi-ideal,CS}, so $b_1<_Sb_2.$  Therefore, the chain 
$$R_{b_2}^S\leq_S R_{b_3}^S\leq_S\cdots\leq_S R_{b_{3n-1}}^S$$
has size at most $n-1.$  Since $\frac{3n-2}{n-1}>3,$ by the generalised pigeonhole principle we obtain $b_i\,\ar_S\,b_{i+3}$ for some $i\in\{2, \dots, 3n-4\}.$  But then the same argument as that of Theorem \ref{thm:bi-ideal,bound} yields a contradiction.
\end{proof}

\begin{cor}
\label{cor:bi-ideal,CS,bound}
Let $S$ be a semigroup with finite $\mathcal{R}$-height whose kernel is completely simple, and let $B$ be a bi-ideal of $S.$  Then $\HR(B)\leq 3\HR(S)-2.$
\end{cor}

If $B$ is a bi-ideal in which every element has a local right identity, then it follows from Lemma \ref{lem:lri} that there exists a chain of $\ar_B$-classes of length $i$ if and only if there exists a chain of $\ar_S$ classes that intersect $B$ of length $i.$  Thus we deduce:

\begin{prop}
\label{prop:lri}
Let $S$ be a semigroup with finite $\mathcal{R}$-height, and let $B$ be a bi-ideal of $S$ in which every element has a local right identity in $B.$  Then $$\HR(B)=n,$$ where $n$ is the maximal length of a chain of $\ar$-classes of $S$ that intersect $B.$ 
\end{prop}

The next result provides a bound on the $\ar$-height of a right ideal of a semigroup with finite $\ar$-height.

\begin{thm}
\label{thm:rightideal,bound}
Let $S$ be a semigroup with finite $\mathcal{R}$-height, and let $A$ be a right ideal of $S.$  Then $$\HR(A)\leq 2n-1,$$ where $n$ is the maximum length of a chain of $\ar$-classes of $S$ contained in $A.$
\end{thm}

\begin{proof}
Suppose first that $n=1.$  Then $A$ is a union of minimal right ideals of $S.$  Minimal right ideals are right simple subsemigroups by \cite[Theorem 2.4]{Clifford:1948}.  It follows that $A$ is a union of minimal right ideals of itself, and hence $\HR(A)=1$ by Lemma \ref{R-height1}. 
 
Now assume that $n>1.$  Suppose for a contradiction that $\HR(A)\geq 2n.$  Then, by Lemma \ref{bi-ideal chain}, there exists a chain 
$$a_1<_A a_2<_A\cdots<_A a_{2n}$$
where $a_1\in K(S).$  Then we have a chain
$$a_1\leq_S a_2\leq_S\cdots\leq_S a_{2n}.$$
By assumption, the chain 
$$R_{a_1}^S\leq_S R_{a_2}^S\leq_S\cdots\leq_S R_{a_{2n}}^S$$
has size at most $n.$
Since $a_1\in K,$ we have $a_1\,\ar_S\,a_1^2,$ so that $a_1\in a_1^2S^1\in a_1A,$ using the fact that $A$ is a right ideal of $S.$  We must then have $a_1<_S a_2,$ for otherwise we would have $a_2\in a_1S\subseteq a_1AS\subseteq a_1A.$  It follows that the chain 
$$R_{a_2}^S\leq_S R_{a_3}^S\leq_S\cdots\leq_S R_{a_{2n}}^S$$
has size at most $n-1.$
Since $\frac{2n-1}{n-1}=2+\frac{1}{n-1}>2,$ by the generalised pigeonhole principle we deduce that there exists $i\in\{1, \dots, 2n-2\}$ such that $a_i\,\ar_S\,a_{i+2}.$  We then have 
$$a_{i+2}\in a_iS\subseteq a_{i+1}AS\subseteq a_{i+1}A,$$
using the fact that $A$ is a right ideal of $S.$  But this contradicts that $a_{i+1}<_Aa_{i+2}$.  
\end{proof}

\begin{cor}
\label{cor:rightideal,bound}
Let $S$ be a semigroup with finite $\mathcal{R}$-height, and let $A$ be a right ideal of $S.$  Then $\HR(A)\leq 2\HR(S)-1.$
\end{cor}

We now turn our attention to left ideals.

\begin{thm}
\label{thm:leftideal,bound}
Let $S$ be a semigroup with finite $\mathcal{R}$-height, and let $A$ be a left ideal of $S.$  Then $$\HR(A)\leq 2n,$$ where $n$ is the maximal length of a chain of $\ar$-classes of $S$ that intersect $A.$
\end{thm}

\begin{proof}
Suppose for a contradiction that there exists a chain 
$$a_1<_A a_2<_A\cdots<_A a_{2n+1}.$$
Then we have a chain
$$a_1\leq_S a_2\leq_S\cdots\leq_S a_{2n+1}.$$
By assumption, the chain 
$$R_{a_1}^S\leq_S R_{a_2}^S\leq_S\cdots\leq_S R_{a_{2n+1}}^S$$
has size at most $n.$
Since $\frac{2n+1}{n}=2+\frac{1}{n}>2,$ by the generalised pigeonhole principle we obtain $i\in\{1, \dots, 2n-1\}$ such that $a_i\,\ar_S\,a_{i+2}.$  Thus $a_{i+2}\in a_iS.$
We then have
$$a_{i+1}\in a_{i+2}A\subseteq a_iSA\subseteq a_iA,$$
using the fact that $A$ is a left ideal of $S.$ But this contradicts that $a_i<_Aa_{i+1}$.
\end{proof}

\begin{cor}
\label{cor:leftideal,bound}
Let $S$ be a semigroup with finite $\mathcal{R}$-height, and let $A$ be a left ideal of $S.$  Then $\HR(A)\leq 2\HR(S).$
\end{cor}

Again, we obtain a slightly shorter bound in the case that $K(S)$ is completely simple.

\begin{thm}
\label{thm:leftideal,CS,bound}
Let $S$ be a semigroup with finite $\mathcal{R}$-height whose kernel is completely simple, and let $A$ be a left ideal of $S.$  Then $$\HR(A)\leq 2n-1,$$ where $n$ is the maximal length of a chain of $\ar$-classes of $S$ that intersect $A.$
\end{thm}

\begin{proof}
Suppose for a contradiction that $\HR(A)\geq 2n.$  Then, by Lemma \ref{bi-ideal chain}, there exists a chain 
$$a_1<_A a_2<_A\cdots<_A a_{2n},$$
where $a_1\in K(S).$
Then we have a chain
$$a_1\leq_S a_2\leq_S\cdots\leq_S a_{2n}.$$
By assumption, the chain 
$$R_{a_1}^S\leq_S R_{a_2}^S\leq_S\cdots\leq_S R_{a_{2n}}^S$$
has size at most $n.$  We cannot have $a_1\,\ar_S\,a_2,$ since that would imply that $a_1\,\ar_A\,a_2$ by Corollary \ref{cor:bi-ideal,CS}, so $a_1<_Sa_2.$  Therefore, the chain 
$$R_{a_2}^S\leq_S R_{a_3}^S\leq_S\cdots\leq_S R_{a_{2n}}^S$$
has size at most $n-1.$
Since $\frac{2n-1}{n-1}>2,$ by the generalised pigeonhole principle we obtain $i\in\{2, \dots, 2n-2\}$ such that $a_i\,\ar_S\,a_{i+2}.$  But then the same argument as that of Theorem \ref{thm:leftideal,bound} yields a contradiction.
\end{proof}

\begin{cor}
\label{cor:leftideal,CS,bound}
Let $S$ be a semigroup with finite $\mathcal{R}$-height whose kernel is completely simple, and let $A$ be a left ideal of $S.$  Then $\HR(A)\leq 2\HR(S)-1.$
\end{cor}

If a left ideal $A$ of $S$ is contained in the set $\text{Reg}(S)$ of regular elements,  then for any $a\in S$ we have $a\in aSa\in aA,$ using the fact that $a$ is regular and $A$ is a left ideal, so every element of $A$ has a local right identity in $A.$  Thus, by Proposition \ref{prop:lri}, we have:

\begin{prop}
\label{prop:left ideal,reg}
Let $S$ be a semigroup with finite $\mathcal{R}$-height, and let $A$ be a left ideal of $S$ such that $A\subseteq\emph{Reg}(S).$  Then $$\HR(A)=n,$$ where $n$ is the maximal length of a chain of $\ar$-classes of $S$ that intersect $A.$ 
\end{prop}

Proposition \ref{prop:left ideal,reg} does not hold if we replace `left ideal' by `right ideal', as the following example demonstrates.

\begin{ex}
\label{ex:Brandt}
Let $S$ be the semigroup with universe $\{(1, 1), (1, 2), (2, 1), (2, 2), 0\}$ and multiplication given by 
$$(i, j)(k, l)=
\begin{cases}
(i, l)&\text{ if }j=k\\
0&\text{ otherwise,}
\end{cases}$$
and $0(i, j)=(i, j)0=00=0.$  Then $S$ is a completely 0-simple inverse semigroup, where the inverse of each $(i, j)$ is $(j, i).$  Consider the right ideal 
$$A=(1, 1)S^1=\{(1, 1), (1, 2), 0\}.$$
Certainly $A\subseteq\text{Reg}(S)=S$.  It is straightforward to verify that the posets of $\ar$-classes of $S$ and $A$ are as presented in Figure \ref{fig:Brandt} below, and hence $\HR(S)=2$ and $\HR(A)=3.$
\end{ex}

\begin{center}
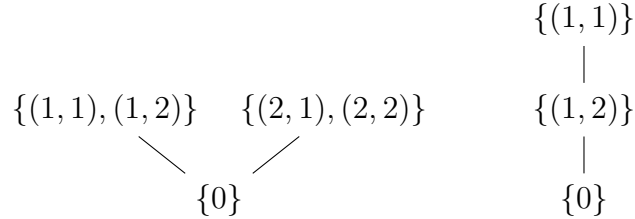
\begin{figure}[H]
\begin{tikzpicture}[scale=.6]
  \node (a) at (-2.5,2) {$\{(1, 1), (1, 2)\}$};
  \node (b) at (2.5,2) {$\{(2, 1), (2, 2)\}$};
  \node (zero) at (0,0) {$\{0\}$};
  \draw (a) -- (zero) -- (b);
\node (c) at (8,4) {$\{(1, 1)\}$};
  \node (d) at (8,2) {$\{(1, 2)\}$};
  \node (0) at (8,0) {$\{0\}$};
\draw (c) -- (d) -- (0) ;
\end{tikzpicture}
\caption{The poset of $\ar_S$-classes (left) and the poset of $\ar_A$-classes (right).}\label{fig:Brandt}
\end{figure}
\vspace{-2.5em}
\end{center}

We now provide a bound on the $\ar$-height of an ideal.

\begin{thm}
\label{thm:ideal,bound}
Let $S$ be a semigroup with finite $\mathcal{R}$-height, and let $A$ be an ideal of $S.$  Then $$\HR(A)\leq n,$$ where $n$ is the maximum length of a chain of $\ar$-classes of $S$ contained in $A.$
\end{thm}

\begin{proof}
Suppose for a contradiction that there exists a chain 
$$a_1<_A a_2<_A\cdots<_A a_{n+1}.$$
Then we have a chain 
$$a_1\leq_S a_2\leq_S\cdots\leq_S a_{n+1}.$$
By assumption, the chain 
$$R_{a_1}^S\leq_S R_{a_2}^S\leq_S\cdots\leq_S R_{a_{n+1}}^S$$
has size at most $n.$
By the pigeonhole principle there exists $i\in\{1, \dots, n\}$ such that $a_i\,\ar_S\,a_{i+1}.$  Then we have
$$a_{i+1}\in a_iS\subseteq a_{i+1}AS\subseteq a_iSAS\subseteq a_iA,$$
using the fact $A$ is an ideal of $S.$ But this contradicts that $a_i<_Aa_{i+1}.$ 
\end{proof}

\begin{cor}
\label{cor:ideal,bound}
Let $S$ be a semigroup with finite $\mathcal{R}$-height, and let $A$ be an ideal of $S.$  Then $\HR(A)\leq\HR(S).$
\end{cor}

We conclude this section by demonstrating that the $\ar$-height of an ideal can be substantially lower than the bound provided in Theorem \ref{thm:ideal,bound}.

\begin{ex}
For any $n\in\N,$ let $T$ be a semigroup such that $\HR(T)=n.$  Let $N=\{x_a : a\in T\}\cup\{0\}$ be a set disjoint from $T.$  We turn $N$ into a null semigroup by defining $xy=0$ for all $x, y\in N.$  Let $S=T\cup N,$ and define a multiplication on $S,$ extending those on $T$ and $N,$ as follows:
$$ax_b=x_ab=x_{ab}\;\text{ and }\;a0=0a=0$$
for all $a, b\in T.$  Then $N$ is an ideal of $S.$  It is straightforward to show that for any $a, b\in T,$ we have $a<_Tb$ if and only if $x_a<_Sx_b$; and clearly $0<_S x_a.$  It follows that the maximum length of a chain of $\ar$-classes of $S$ contained in $N$ is $n+1.$  On the other hand, it is easy to see that $\HR(N)=2.$
\end{ex}

\section{\large{Attaining the Bounds}\nopunct}
\label{sec:attainbounds}

Given Theorems \ref{thm:bi-ideal,bound}-\ref{thm:ideal,bound}, an immediate question arises: Can the bounds established in these results be attained?  In fact, one can ask a stronger question: Can these bounds be attained for every natural number $n$ and in such a way that $\HR(S)=n$?  More precisely, we have the following problems.
\begin{enumerate}[leftmargin=*]
\item For each $n\in\mathbb{N},$ does there exist a semigroup $S$ with a bi-ideal $B$ such that $\HR(S)=n$ and $\HR(B)=3n-1$?
\item For each $n\in\mathbb{N},$ does there exist a semigroup $S$ with a completely simple kernel and a bi-ideal $B$ such that $\HR(S)=n$ and $\HR(B)=3n-2$?
\item For each $n\in\mathbb{N},$ does there exist a semigroup $S$ with a right ideal $A$ such that $\HR(S)=n$ and $\HR(A)=2n-1$?
\item For each $n\in\mathbb{N},$ does there exist a semigroup $S$ with a left ideal $A$ such that $\HR(S)=n$ and $\HR(A)=2n$?
\item For each $n\in\mathbb{N},$ does there exist a semigroup $S$ with a completely simple kernel and a left ideal $A$ such that $\HR(S)=n$ and $\HR(A)=2n-1$?
\item For each $n\in\mathbb{N},$ does there exist a semigroup $S$ with an ideal $A$ such that $\HR(S)=\HR(A)=n$?
\end{enumerate}

Unfortunately, we have not been able to answer question (1).

\begin{prob}
For each $n\in\mathbb{N},$ does there exist a semigroup $S$ with a bi-ideal $B$ such that $\HR(S)=n$ and $\HR(B)=3n-1$?
\end{prob}

We shall answer questions (2)-(6) in the positive.  Question (6) is easily dealt with: take $S$ to be any semigroup with $\HR(S)=n$ and set $A=S.$\par
We now consider question (2).  The case $n=1$ is trivial: we can just take $S$ to be the trivial semigroup and $B=S$; then $\HR(S)=\HR(B)=1=3(1)-2.$  For $n\geq 2,$ the following result provides the desired semigroups.

\begin{thm}
\label{thm:bi-ideal}
Let $n\geq 2.$  Let $S$ be the finite semigroup defined by the presentation
\begin{align*}
\langle x, y, z, t\,|\,&xyzt=x,\, yzty=y,\, ztyz=z,\, tyzt=t, w=0\\
&(w\in\{x^n,\, y^2,\, z^2,\, t^2,\, xz,\, xt,\, yx,\, yt,\, zx,\, zy,\, tz,\, tx^{n-1}\})\rangle.
\end{align*}
Let $X=\{x, y, z, tx\}\subseteq S$ and let $B=X\cup XS^1X$ ($B$ is the smallest bi-ideal of $S$ containing $X$).  Then $\HR(S)=n$ and $\HR(B)=3n-2.$
\end{thm}

\begin{proof}
We begin by finding a normal form for $S.$  It is straightforward to show that the associated rewriting system of the presentation for $S$ is complete.  That it is noetherian follows from the fact that all the rewriting rules are length-reducing.  For confluence, it suffices to check that all the critical pairs resolve.  For instance, $(0tyz, xz)$ is a critical pair, since $xztyz\to 0tyz$ and $xztyz\to xz,$ and clearly both sides of this pair reduce to 0.  This rewriting system, therefore, yields the following normal form for $S,$ consisting of all words over $\{x, y, z, t\}$ that do not contain as a subword the left-hand side of any of the rewriting rules, along with 0:
$$\{x^i,\, x^iy,\, x^iyz : 1\leq i\leq n-1\}\cup\{y,\, yz,\, yzt,\, z,\, zt,\, zty,\, t,\, ty,\, tyz\}\cup U\cup\{0\},$$
where 
$$U=\{yztx^i,\, yztx^iy,\, yztx^iyz,\, ztx^i,\, ztx^iy,\, ztx^iyz,\, tx^i,\, tx^iy,\, tx^iyz : 1\leq i\leq n-2\}.$$
We note that $U=\emptyset$ if $n=2.$
It is easy to calculate that $|S|=12(n-1)+1.$  Using the above normal form and the relations of the presentation, it is easy to show that $t, zt, ty, yzt, tyz\notin B.$  All other elements in the normal form have the form $u$ or $uwv$ where $u, v\in X$ and $w\in\{x, y, z, t\}^{\ast},$ so they belong to $B.$  Thus $B=S\!\setminus\!\{t, zt, ty, yzt, tyz\}.$

From the relations of the presentation, it follows that for each generator $u\in\{x, y, z, t\},$ the principal right ideal $uS^1$ consists of precisely the words in the normal form whose first letter is $u,$ along with $0.$  Thus, for any two generators $u, v\in\{x, y, z, t\}, u\neq v,$ there are no elements of the form $uw$ and $vw^{\prime}$ in $S$ such that $uw\,\ar_S\,vw^{\prime}.$

Now let $R_i=\{x^i, x^iy, x^iyz\}$ ($1\leq i\leq n-1$), $S_1=\{y, yz, yzt\},$ $T_1=\{z, zt, zty\}$ and $U_1=\{t, ty, tyz\}.$  It is easy to see from the presentation that each of the sets is an $\ar_S$-class.  For $i\in\{2, \dots, n-1\},$ let 
$$S_i=yztR_{i-1}=\{yztx^{i-1}, yztx^{i-1}y, yztx^{i-1}yz\},$$
$T_i=ztR_{i-1}$ and $U_i=tR_{i-1}.$  Since $\ar_S$ is a left congruence on $S,$ and $R_{i-1}$ is an $\ar_S$-class, each $S_i,$ $T_i$ and $U_i$ is also an $\ar_S$-class.  Of course, $\{0\}$ is both an $\ar_S$-class and an $\ar_B$-class.

Consider $i\in\{1, \dots, n-2\}.$  Clearly $R_i\geq R_{i+1}.$  Also, we have $yztx^{i-1}yz(tx)=yztx^i\in S_{i+1},$ so $S_i\geq S_{i+1}$.  Similarly, $T_i\geq T_{i+1}$ and $U_i\geq U_{i+1}.$  It is easy to see from the presentation that for any $v\in S$ and $s\in S^1$ with $vs\neq 0,$ we have $|v|_x\leq|vs|_x,$ where $|w|_x$ denotes the number of appearances of $x$ in $w.$  Thus $x^i\notin x^{i+1}S^1,$ so $R_i>R_{i+1}.$  Similarly, we have $S_i>S_{i+1},$ $T_i>T_{i+1}$ and $U_i>U_{i+1}.$  We conclude that the poset of $\ar_S$-classes is as displayed in Figure \ref{fig:bi-ideal} below, so that $\HR(S)=n.$  

Turning our attention to $B,$ we have 
$$x^i\geq_B x^iy\geq_B x^iyz\geq_B x^iyz(tx)=x^{i+1}.$$
We certainly have $x^iyz>_B x^{i+1}$ since $x^iyz>_S x^{i+1}.$  Also, it is easy to calculate that
$$x^iyB=\{x^iy,\, x^iyz,\, x^j,\, x^jy,\, x^jyz,\, 0 : i+1\leq j\leq n-1\}$$
and 
$$x^iyzB=\{x^j,\, x^jy,\, x^jyz,\, 0 : i+1\leq j\leq n-1\},$$
so $(x^i, x^iy), (x^iy , x^iyz)\notin\ar_B.$  Thus we have a chain
$$x>_B xy>_B xyz>_B x^2>_B x^2y>_B x^2yz>_B\cdots>_B x^{n-1}>_B x^{n-1}y>_B x^{n-1}yz>_B 0,$$
so $\HR(B)\geq 3n-2.$  By Theorem \ref{thm:bi-ideal,CS,bound}, we have $\HR(B)\leq 3n-2.$  Thus $\HR(B)=3n-2.$
\end{proof}

\begin{center}
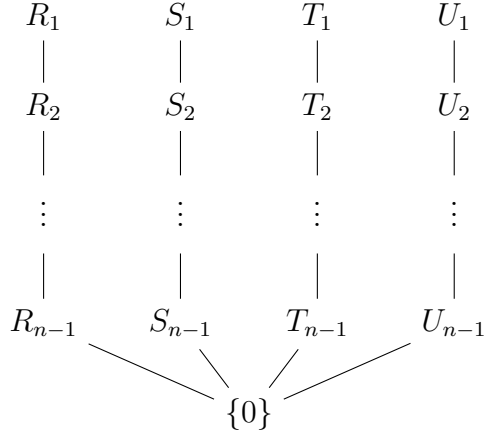
\begin{figure}[H]
\begin{tikzpicture}[scale=.6]
\node (R1) at (-4.5,8.8) {$R_1$};
\node (R2) at (-4.5,6.8) {$R_2$};
\node (a) at (-4.5,4.6) {$\vdots$};
\node (a1) at (-4.5,3.8) {~};
\node (b) at (-4.5,2) {$R_{n-1}$};
\node (c) at (-1.5,2) {$S_{n-1}$};
\node (d1) at (-1.5,3.8) {~};
\node (d) at (-1.5,4.6) {$\vdots$};
\node (S2) at (-1.5, 6.8) {$S_2$};
\node (S1) at (-1.5, 8.8) {$S_1$};
\node (zero) at (0,0) {$\{0\}$};
\node (U1) at (1.5,8.8) {$T_1$};
\node (U2) at (1.5,6.8) {$T_2$};
\node (e1) at (1.5,3.8) {~};
\node (e) at (1.5,4.6){$\vdots$};
\node (f) at (1.5,2) {$T_{n-1}$};
\node (g) at (4.5,2) {$U_{n-1}$};
\node (h1) at (4.5,3.8) {~};
\node (h) at (4.5,4.6) {$\vdots$};
\node (V2) at (4.5, 6.8) {$U_2$};
\node (V1) at (4.5, 8.8) {$U_1$};
\draw (R1) -- (R2) -- (a)
(a1) -- (b) -- (zero) -- (c) -- (d1) 
(d) -- (S2) -- (S1)
(U1) -- (U2) -- (e) 
(e1) -- (f) -- (zero) -- (g) -- (h1)
(h) -- (V2) -- (V1);
\end{tikzpicture}
\caption{The poset of $\ar$-classes of the semigroup $S$ given in the statement of Theorem \ref{thm:bi-ideal}.}
\label{fig:bi-ideal}
\end{figure}
\vspace{-2.5em}
\end{center}

\begin{center}
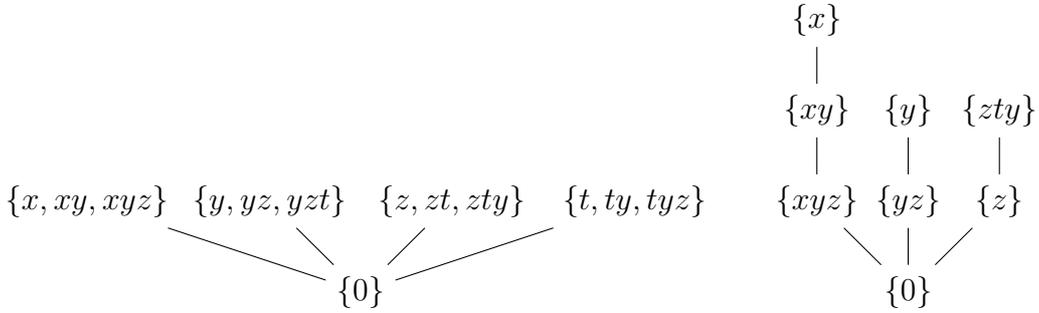
\begin{figure}[H]
\begin{tikzpicture}[scale=.6]
\node (a) at (-6,2) {$\{x, xy, xyz\}$};
\node (b) at (-2,2) {$\{y, yz, yzt\}$};
\node (c) at (2,2) {$\{z, zt, zty\}$};
\node (d) at (6,2) {$\{t, ty, tyz\}$};
\node (zero) at (0,0) {$\{0\}$};
\draw (a) -- (zero) -- (b) 
(c) -- (zero) -- (d);
\node (e) at (12,4) {$\{y\}$};
\node (f) at (12,2) {$\{yz\}$};
\node (g) at (10,6) {$\{x\}$};
\node (h) at (10,4) {$\{xy\}$};
\node (i) at (10,2) {$\{xyz\}$};
\node (j) at (14,2) {$\{z\}$};
\node (k) at (14,4) {$\{zty\}$};
\node (0) at (12,0) {$\{0\}$};
\draw (g) -- (h) -- (i) -- (0) -- (j) -- (k)
(e) -- (f) -- (0);
\end{tikzpicture}
\caption{Let $S$ and $B$ be as given in Theorem \ref{thm:bi-ideal} in the case $n=2.$  The poset of $\ar_S$-classes is displayed on the left, and the poset of $\ar_B$-classes is displayed on the right.}
\label{fig:bi-ideal,n=2}
\end{figure}
\vspace{-1.5em}
\end{center}

We now move on to problem (3).  To solve this, we utilise the following construction.

\begin{defn}
Let $S$ be a semigroup and let $I$ be a non-empty set.  The {\em Brandt extension of $S$ by} $I,$ denoted by $\mathcal{B}(S, I),$ is the semigroup with universe $(I\times S\times I)\cup\{0\}$ and multiplication given by
$$(i, s, j)(k, t, l)=
\begin{cases}
(i, st, l)&\text{ if }j=k\\
0&\text{ otherwise,}
\end{cases}$$
and $0x=x0=0$ for all $x\in(I\times S\times I)\cup\{0\}.$
\end{defn}

\begin{remark}
Brandt extensions of groups, known as {\em Brandt semigroups}, are precisely the completely 0-simple inverse semigroups \cite[Theorem 5.1.8]{Howie:1995}.  The semigroup $S$ from Example \ref{ex:Brandt} is (isomorphic to) the 5-element Brandt semigroup over the trivial group.  We note also that the subsemigroup $\langle y, z, t\rangle$ of the semigroup $S$ from Theorem \ref{thm:bi-ideal} is the 10-element Brandt semigroup over the trivial group.
\end{remark}

\begin{thm}
\label{thm:Brandt}
Let $S$ be a semigroup with finite $\ar$-height, and let $A=aS^1$ be a principal right ideal of $S.$  Let $I$ be any set with $|I|\geq 2,$ and let $T=\mathcal{B}(S, I).$  Fix $1\in I,$ and consider the principal right ideal $B=(1, a, 1)T^1$ of $T.$  Then $\HR(T)=\HR(S)+1$ and $\HR(B)=\HR(A)+2.$
\end{thm}

\begin{proof}
It can be easily proved that for any $s, t\in S$ and $i, j, k, l\in I,$ we have
\begin{equation}
\label{eqtn:Brandt}
(i, s, j)<_T(k, t, l)\Leftrightarrow i=k\text{ and }s<_St.
\end{equation}
Let $\HR(S)=n.$  Then there exists a chain
$$s_1<_Ss_2<_S\cdots<_Ss_n.$$
Letting $t_i=(1, s_i, 1),$ by (\ref{eqtn:Brandt}) we have a chain
$$0<_Tt_1<_Tt_2<_T\cdots<_Tt_n,$$
so $\HR(T)\geq n+1.$  Now suppose for a contradiction that
$\HR(T)>n+1.$  Then there exists a chain
$$0<_Tx_1<_Tx_2<_T\cdots<_Tx_{n+1}.$$  By (\ref{eqtn:Brandt}) there exists $i\in I$ such that each $x_k$ has the form $(i, y_k, j_k)$ for some $y_k\in S$ and $j_k\in I.$  But then, by (\ref{eqtn:Brandt}), we have a chain
$$y_1<_Sy_2<_S\cdots<_Sy_{n+1},$$
contradicting the fact that $\HR(S)=n.$  Thus $\HR(T)=n+1.$\par
Now let $\HR(A)=m.$  Then there exists a chain
$$a_1<_Aa_2<_A\cdots<_Aa_m.$$
We have that $B=\{(1, a, 1)\}\cup\{(1, as, i) : s\in S, i\in I\}\cup\{0\}.$
Let $b_i=(1, a_i, 1).$  Then $b_i\in B.$  For each $i\in\{1, \dots, m\},$ there exists $c_i\in A$ such that $a_i=a_{i+1}c_i.$  Therefore, we have that $b_i=b_{i+1}(1, c_i, 1)\in b_{i+1}B,$ so $b_i\leq_B b_{i+1}.$  Clearly, if $b_i\,\ar_B\,b_{i+1}$ then $a_i\,\ar_A\,a_{i+1},$ contradicting that $a_i<_Aa_{i+1},$ so $b_i<_Bb_{i+1}.$  Now choose $s\in S$ and $j\in I\!\setminus\!\{1\},$ and let $b_0=(1, a_1(as), j).$  Then $b_0=b_1(1, as, j)\in b_1B.$ Since $j\neq 1,$ we have that $b_0B=\{0\},$ so $b_0<_Bb_1.$  Clearly $0<_B b_0.$  In conclusion, we have a sequence
$$0<_Bb_0<_Bb_1<_Bb_2<_B\cdots<_Bb_m,$$
so $\HR(B)\geq m+2.$  Now suppose for a contradiction that
$\HR(B)>m+2.$  Then there exists a chain
$$0<_Bd_1<_Bd_2<_B\cdots<_Bd_{m+2}.$$  Let $d_i=(1, c_i, j_i).$  Then we have a chain
$$c_1\leq_Ac_2\leq_A\cdots\leq_Ac_{m+2}.$$
Since $\HR(A)=m,$ it follows that there exist $k, l\in\{1, \dots, m+1\}$ with $k<l$ such that $c_k\,\ar_A\,c_{k+1}$ and $c_l\,\ar_A\,c_{l+1}.$  In particular, we have $c_{l+1}\in c_lA^1.$  Since $d_l<_B d_{l+1},$ we must have that $c_l\in c_{l+1}A,$ and hence 
$$c_{l+1}\in c_lA^1\subseteq c_{l+1}AA^1\subseteq c_lA^1AA^1=c_lA.$$ 
So, there exists $u\in A$ such that $c_{l+1}=c_lu.$  We cannot have $j_l=1,$ for then $d_{l+1}=d_l(1, u, j_{l+1})\in d_lB,$ contradicting that $d_l<_Bd_{l+1}.$  But then $d_lB=\{0\},$ contradicting that $d_k\in d_lB.$  We conclude that $\HR(B)=m+2.$
\end{proof}

\begin{cor}
For any $n\in\mathbb{N},$ there exists a finite semigroup $S$ with a principal right ideal $A$ such that $\HR(S)=n$ and $\HR(A)=2n-1.$
\end{cor}

\begin{proof}
We prove the result by induction.  For $n=1,$ take $S$ to be the trivial semigroup and $A=S.$
Now let $n\geq 1,$ and assume that there exists a finite semigroup $S$ with a principal right ideal $A=aS^1$ such that $\HR(S)=n$ and $\HR(A)=2n-1.$  Let $T$ and $B$ be as given in the statement of Theorem \ref{thm:Brandt}.  Then, by Theorem \ref{thm:Brandt}, we have that $\HR(T)=n+1$ and $\HR(B)=(2n-1)+2=2(n+1)-1.$  This completes the proof.
\end{proof}

We now turn our attention to question (4), beginning with the case $n=1.$

\begin{prop}
\label{prop:rightsimple}
Let $S$ be a right simple semigroup with no idempotent (e.g.\ a Baer-Levi semigroup).  (Then $\HR(S)=1.$)  Let $a\in S$ be arbitrary, and consider the principal left ideal $A=S^1a.$  Then the $\ar_A$-classes are $\{a\}$ and $A\!\setminus\!\{a\}=Sa,$ and hence $\HR(A)=2.$
\end{prop}

\begin{proof}
Let $s, t\in S.$  Since $S$ is right simple, there exists $x\in S^1$ such that $s=(ta)x.$  Thus $sa=(ta)(xa)\in(ta)A.$  Similarly, we have $ta\in(sa)A,$ so $sa\,\ar_A\,ta.$  Now suppose for a contradiction that $a\,\ar_A\,ua$ for some $u\in S.$  Then, together with the fact just proved that $ua\,\ar_A\,a^2,$ we have $a\,\ar_A\,a^2$ by transitivity.  Therefore, there exists $y\in S^1$ such that $a=a^2(ya).$  But then $a^2y$ is an idempotent, so we have a contradiction.  Thus the $\ar_A$-classes are $\{a\}$ and $Sa.$  Clearly $a^2<_A a,$ so we conclude that $\HR(A)=2.$
\end{proof} 

\begin{thm}
\label{thm:leftideal}
Let $S$ be a semigroup with finite $\ar$-height, and let $A$ be a left ideal of $S.$  Let $T$ be any right simple semigroup with no idempotent, and let $U$ be the semigroup defined by the presentation
$$\langle S\cup T\,|\,ab=a\cdot b,\, cd=c\cdot d,\; ac=c\;\, (a, b\in S, c, d\in T)\rangle.$$
Fix $c\in T,$ and let $B$ denote the left ideal $T^1(A\cup\{c\})$ of $U.$
Then $\HR(U)=\HR(S)+1$ and $\HR(B)=\HR(A)+2.$
\end{thm}

\begin{proof}
The semigroup $U$ has a normal form $S\cup T\cup TS.$  Let $K=T\cup TS.$  It is easy to see that $K$ is an ideal of $U.$  All elements of $T$ are $\ar_U$-related since $T$ is right simple, and for any $a\in S$ and $t\in T$ we have that $ta=t\cdot a$ and $(ta)T=tT=T.$  Thus $K$ is an $\ar$-class of $U.$  It follows that $K$ is the minimal ideal of $U.$  Now, since $U\!\setminus\!S=K$ is an ideal, it follows that the restriction of $\leq_U$ to $S$ is $\leq_S$.  Thus the poset of $\ar_U$-classes can be viewed as the poset of $\ar_S$-classes along with the minimum element $K.$  This is depicted in Figure \ref{fig:leftideal} below.  It follows that $\HR(U)=\HR(S)+1.$

We now consider the left ideal $B$ of $U.$  We have $B\cap S=A.$  Since $B\!\setminus\!A$ is an ideal, the restriction of $\leq_B$ to $A$ is $\leq_A$.  
We claim that the sets $\{c\}$ and $TA\cup Tc$ are $\ar$-classes of $B.$
First, let $t, t^{\prime}\in T$ and $a, a^{\prime}\in A.$  Since $T$ is right simple, there exist $x, y, z\in T$ such that
$$t=t^{\prime}x,\; t^{\prime}=ty,\; t=(t^{\prime}c)z.$$ 
(We can assume that $x, y, z\in T$ even if $t=t^{\prime}$ or $t=t^{\prime}c$, since every element of $T$ has a local right identity.)
Using the defining relations $a^{\prime}x=x$ and $ay=y,$ we deduce that
$$ta=(t^{\prime}a^{\prime})(xa),\; t^{\prime}a^{\prime}=(ta)(ya)\;\text{ and }\;ta=(t^{\prime}c)(za),\; t^{\prime}c=(ta)(yc);$$
so $ta\,\ar_B\,t^{\prime}a^{\prime}$ and $ta\,\ar_B\,t^{\prime}c.$
Since $t, t^{\prime}, a, a^{\prime}$ were chosen arbitrarily, it follows by transitivity that all elements in $TA\cup Tc$ are $\ar_B$-related.  Now suppose for a contradiction that $c\,\ar_B\,c^2.$  Then there exists $b\in B$ such that $c=c^2b.$  We cannot have $b\in T^1A,$ for this would imply that $c\in T\cap TS,$ contradicting the fact that $S\cup T\cup TS$ is a normal form for $U.$  Thus $b=wc$ for some $w\in T^1,$ and hence $c=c^2wc.$  But then $c^2w$ is an idempotent of $T,$ so we have a contradiction.  This proves the claim.

For any $a\in A$ we have $ac=c,$ so $c<_B a.$  Also, we have $c^2<_B c.$  Thus the poset of $\ar_B$-classes can be viewed as the poset of $\ar_A$-classes along with the elements $\{c\}$ and $TA\cup TC,$ where $\{c\}$ is below all the $\ar_A$-classes and $TA\cup TC$ is the minimum element; see Figure \ref{fig:leftideal} for an illustration.  It follows that $\HR(B)=\HR(A)+2.$
\end{proof}

\begin{center}
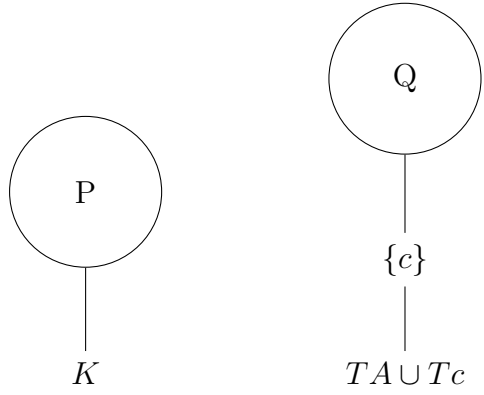
\begin{figure}[H]
\begin{tikzpicture}[scale=.6]
  \node[circle,draw,inner sep=0pt,minimum size=2cm] (a) at (0,4) {P};
  \node (b) at (0,0) {$K$};
  \draw (a) -- (b);
  \node[circle,draw,inner sep=0pt,minimum size=2cm] (c) at (7,6.5) {Q};
  \node (d) at (7,2.5) {$\{c\}$};
  \node (e) at (7,0) {$TA\cup Tc$};
\draw (c) -- (d) -- (e) ;
\end{tikzpicture}
\caption{Let $S,$ $A,$ $U$ and $B$ be as given in Theorem \ref{thm:leftideal}.  The poset of $\ar_U$-classes is displayed on the left, where $P$ denotes the poset of $\ar_S$-classes.  The poset of the $\ar_B$-classes is displayed on the right, where $Q$ denotes the poset of $\ar_A$-classes.}
\label{fig:leftideal}
\end{figure}
\vspace{-1.5em}
\end{center}

\begin{cor}
For any $n\in\mathbb{N},$ there exists a semigroup $S$ with a left ideal $A$ such that $\HR(S)=n$ and $\HR(A)=2n.$
\end{cor}

\begin{proof}
We prove the result by induction.  Proposition \ref{prop:rightsimple} deals with the base case.
Now let $n\geq 1,$ and assume that there exists a semigroup $S$ with a left ideal $A$ such that $\HR(S)=n$ and $\HR(A)=2n.$  Let $U$ and $B$ be as given in the statement of Theorem \ref{thm:leftideal}.  Then, by Theorem \ref{thm:leftideal}, we have $\HR(U)=n+1$ and $\HR(B)=2n+2=2(n+1).$  This completes the proof.
\end{proof}

Finally, we solve problem (5) with the following result, the case $n=1$ being trivial.

\begin{thm}
\label{thm:leftideal,CS}
Let $n\geq 2.$  Let $S$ be the finite semigroup defined by the presentation
$$\langle x, y, z\,|\,xyz=x,\, yzy=y,\, zyz=z,\, w=0\;\, (w\in\{x^n,\, y^2,\, z^2,\, xz,\, yx,\, zx^{n-1}\})\rangle,$$
and let $A=S^1\{x, y\}.$  Then $\HR(S)=n$ and $\HR(A)=2n-1.$
\end{thm}

\begin{proof}
The proof of this result is similar to that of Theorem \ref{thm:bi-ideal}, so we will not go into as much detail.\par
The associated rewriting system of the presentation for $S$ is complete, yielding the following normal form for $S$:
$$\{x^i,\, x^iy : 1\leq i\leq n-1\}\cup\{yzx^j,\, yzx^jy,\, zx^j,\,  zx^jy  : 0\leq j\leq n-2\}\cup\{0\},$$
It is straightforward to calculate that $|S|=6(n-1)+1.$  It can also be easily shown that $A=S\!\setminus\!\{z, yz\}.$\par
Let $R_i=\{x^i, x^iy\}$ for $i\in\{1, \dots, n-1\}.$  Let $S_1=\{y, yz\}$ and $T_1=\{z, zy\},$ and for $i\in\{2, \dots, n-1\}$ let 
$S_i=yzR_{i-1}$ and $T_i=zR_{i-1}.$  Then each $R_i$, $S_i$ and $T_i$ is an $\ar_S$-class.  The poset of $\ar_S$-classes is as displayed in Figure \ref{fig:leftideal,CS} below, so that $\HR(S)=n.$  Turning our attention to $A,$ we have
$$x^i\geq_A x^iy\geq_A x^iy(zx)=x^{i+1}.$$
We certainly have $x^iy>_A x^{i+1}$ since $x^iy>_S x^{i+1}.$  Also, we have 
$$x^iyA=\{x^iy,\, x^j,\, x^jy,\, 0 : i+1\leq j\leq n-1\},$$
so $x^i>_A x^iy.$  Thus we have a chain
$$x>_A xy>_A x^2>_A x^2y>_A\cdots>_A x^{n-1}>_A x^{n-1}y>_A 0,$$
so $\HR(A)\geq 2n-1.$  By Theorem \ref{thm:leftideal,CS,bound}, we have $\HR(A)\leq 2n-1.$  We conclude that $\HR(A)=2n-1.$
\end{proof}

\begin{center}
\begin{figure}[H]
\begin{tikzpicture}[scale=.6]
\node (R1) at (-4,8.8) {$R_1$};
\node (R2) at (-4,6.8) {$R_2$};
\node (a) at (-4,4.6) {$\vdots$};
\node (a1) at (-4,3.8) {~};
\node (b) at (-4,2) {$R_{n-1}$};
\node (c) at (0,2) {$S_{n-1}$};
\node (d1) at (0,3.8) {~};
\node (d) at (0,4.6) {$\vdots$};
\node (S2) at (0, 6.8) {$S_2$};
\node (S1) at (0, 8.8) {$S_1$};
\node (zero) at (0,0) {$\{0\}$};
\node (T1) at (4,8.8) {$T_1$};
\node (T2) at (4,6.8) {$T_2$};
\node (e1) at (4,3.8) {~};
\node (e) at (4,4.6){$\vdots$};
\node (f) at (4,2) {$T_{n-1}$};
\draw (R1) -- (R2) -- (a)
(a1) -- (b) -- (zero) -- (c) -- (d1) 
(d) -- (S2) -- (S1)
(T1) -- (T2) -- (e) 
(e1) -- (f) -- (zero);
\end{tikzpicture}
\caption{The poset of $\ar$-classes of the semigroup $S$ given in the statement of Theorem \ref{thm:leftideal,CS}.}
\label{fig:leftideal,CS}
\end{figure}
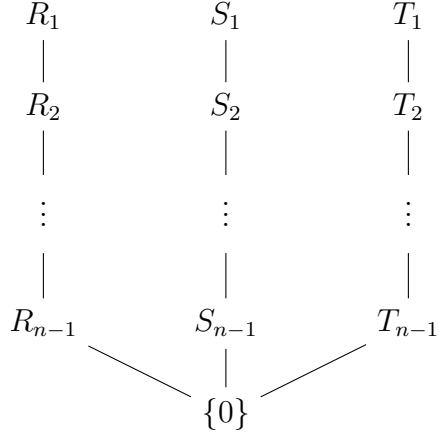
\vspace{-2em}
\end{center}

\begin{center}
\begin{figure}[H]
\begin{tikzpicture}[scale=.6]
\node (a) at (-3,2) {$\{x, xy\}$};
\node (b) at (0,2) {$\{y, yz\}$};
\node (c) at (3,2) {$\{z, zy\}$};
\node (zero) at (0,0) {$\{0\}$};
\draw (a) -- (zero) -- (b)
(zero) -- (c);
\node (d) at (7,4) {$\{x\}$};
\node (e) at (7,2) {$\{xy\}$};
\node (f) at (9,2) {$\{y\}$};
\node (g) at (11,2) {$\{zy\}$};
\node (0) at (9,0) {$\{0\}$};
\draw (d) -- (e) -- (0) -- (f)
(g) -- (0);
\end{tikzpicture}
\caption{Let $S$ and $A$ be as given in Theorem \ref{thm:leftideal,CS} in the case $n=2.$  The poset of $\ar_S$-classes is displayed on the left, and the poset of $\ar_A$-classes is displayed on the right.}\label{fig:leftideal,CS,n=2}
\end{figure}
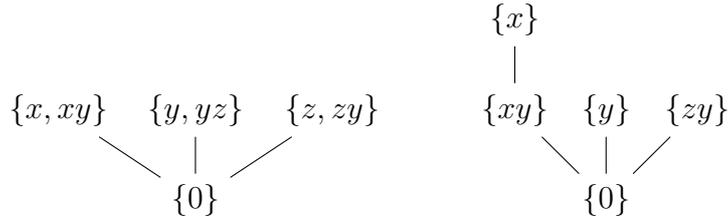
\vspace{-1.5em}
\end{center}

\section{Open Problems and Future Research}
\label{sec:conclusion}

As with the $\ar$-height, one can of course define the $\el$-height, $\h$-height and $\jay$-height of a semigroup $S,$ which we denote by $\HL(S),$ $\HH(S)$ and $\HJ(S),$ respectively.  It would potentially be interesting to consider the relationship between these heights.  We note that for stable semigroups $S,$ since any two $\ar$-classes within the same $\jay(=\dee)$-class are incomparable, we have $\HR(S)\leq\HJ(S).$  It is easy to find stable semigroups for which $\HR(S)=\HJ(S).$  Indeed, for any finite full transformation semigroup $S=\mathcal{T}_n$, we have $\HR(S)=\HL(S)=\HH(S)=\HJ(S)=n.$  On the other hand, if $S$ is the semigroup from Theorem \ref{thm:leftideal,CS} (which is stable since it is finite), then $\HR(S)=n$ by that theorem, but it turns out that $\HJ(S)=2n-1$; indeed, the $\jay$-classes of $S$ form a chain
$$J_1>K_1>\cdots>J_{n-1}>K_{n-1}>\{0\},$$
where $J_i=\{yzx^{i-1}, yzx^{i-1}y, zx^{i-1}, zx^{i-1}y\}$ and $K_i=\{x^i, x^iy\}$ for $1\leq i\leq n-1.$
However, it is not the case that $\HR(S)\leq\HJ(S)$ for every semigroup $S.$  For example, for the bicyclic monoid $B$ we have $\HJ(B)=1$ and $\HR(B)=\infty.$  We raise the following question.

\begin{prob}
Is there a general upper bound for $\HJ(S)$ in terms of $\HR(S)$?
\end{prob}

It is perhaps also worth considering the relationship between the $\jay$-height of a semigroup and that of its bi-ideals, one-sided ideals, etc.
In particular, we ask:

\begin{prob}
Is the property of having finite $\jay$-height inherited by bi-ideals?
\end{prob}

Another possible direction for future research would be to study the $\ar$-height more systematically.  In particular, one could consider the $\ar$-height of semigroups lying in certain special classes, such as regular semigroups.  We note that for an inverse semigroup $S,$ we have $\HR(S)(=\HL(S)=\HH(S))$ is equal to the height of the semilattice $E(S)$ of idempotents of $S.$  Moreoever, one could investigate the behaviour of the $\ar$-height under various semigroup-theoretic constructions, such as quotients, ideal extensions, direct products, free products, etc.

\section*{Acknowledgements}
This work was supported by the Engineering and Physical Sciences Research Council [EP/V002953/1].  The author would like to thank the referee for a careful reading of the paper, and for observations and questions that led to Section \ref{sec:conclusion}.

\end{document}